\documentclass[journal]{IEEEtran}
\usepackage{cite}
\usepackage{amsmath}
\usepackage{amssymb}
\usepackage{amsthm}
\usepackage{arydshln}
\usepackage{mathtools}
\usepackage{mathrsfs}
\usepackage{algorithmic}
\usepackage{graphicx}
\usepackage{epsfig,graphics}
\usepackage{epstopdf}
\usepackage{subfigure}
\usepackage{enumitem}
\usepackage{pifont}
\usepackage{color}
\usepackage{xcolor}

\usepackage{booktabs}
\usepackage{url}

\usepackage{multirow,ragged2e}
\usepackage{bm}
\usepackage{fixltx2e}
\usepackage[english]{babel}
\theoremstyle{theorem}
\newtheorem{theorem}{Theorem}
\theoremstyle{lemma}

\theoremstyle{corollary}
\newtheorem{corollary}{Corollary}
\theoremstyle{definition}

\theoremstyle{assumption}

\theoremstyle{problem}

\theoremstyle{example}

\theoremstyle{proposition}
\newtheorem{proposition}{Proposition}

\newtheorem{remark}{Remark}

\newcommand{\sect}[1]{Section~\ref{#1}}

\newcommand{\eqs}[2]{Equations~(\ref{#1})-(\ref{#2})}
\newcommand{\eqsref}[2]{(\ref{#1})-(\ref{#2})}

\newcommand{\coro}[1]{Corollary~\ref{#1}}

\newcommand{\fbm}[1]{\mathbf{#1}}
\newcommand{\tbm}[1]{\fbm{#1}^\mathsf{T}}
\newcommand{\tfbm}[1]{\bm{#1}^\mathsf{T}}
\newcommand{\ibm}[1]{\fbm{#1}^{-1}}

\newcommand{\hbm}[1]{\hat{\fbm{#1}}}

\newcommand{\thbm}[1]{\hat{\fbm{#1}}^\mathsf{T}}
\newcommand{\tilbm}[1]{\tilde{\fbm{#1}}}

\newcommand{\ttilbm}[1]{\tilde{\fbm{#1}}^\mathsf{T}}

\newcommand{\dottbm}[1]{\dot{\fbm{#1}}^\mathsf{T}}
\newcommand{\dotbm}[1]{\dot{\fbm{#1}}}

\newcommand{\dothatbm}[1]{\dot{\hat{\fbm{#1}}}}

\newcommand{\dothattbm}[1]{\dot{\hat{\fbm{#1}}}^\mathsf{T}}

\newcommand{\ddotbm}[1]{\ddot{\fbm{#1}}}

\newcommand{\ddothbm}[1]{\ddot{\hbm{#1}}}

\begin{document}

\title{Dynamic Interconnection and Damping Injection for Input-to-State Stable Bilateral Teleoperation}

\author{Yuan~Yang, Daniela Constantinescu,~\IEEEmembership{Member,~IEEE,} and Yang Shi,~\IEEEmembership{Fellow,~IEEE}\thanks{The authors are with the Department of Mechanical Engineering, University of Victoria, Victoria, BC V8W 2Y2 Canada (e-mail: yangyuan@uvic.ca; danielac@uvic.ca; yshi@uvic.ca).}}

\maketitle

\begin{abstract}
In bilateral teleoperation, the human who operates the master and the environment which interacts with the slave are part of the force feedback loop. Yet, both have time-varying and unpredictable dynamics and are challenging to model. A conventional strategy for sidestepping the demand for their models in the stability analysis is to assume passive user and environment, and to control the master-communications-slave system to be passive as well. This paper circumvents the need to model the user and environment in a novel way: it regards their forces as external excitations for a semi-autonomous force feedback loop, which it outfits with a dynamic interconnection and damping injection controller that renders bilateral teleoperation with time-varying delays exponentially input-to-state stable. The controller uses the position and velocity measurements of the local robot and the delayed position transmitted from the other robot to robustly synchronize the master and slave under the user and environment perturbations. Lyapunov-Krasovskii stability analysis shows that the proposed strategy (i) can confine the position error between the master and slave to an invariant set, and (ii) can drive it exponentially to a globally attractive set. Thus, the dynamic interconnection and damping injection approach has practical relevance for telemanipulation tasks with given precision requirements. 
\end{abstract}

\IEEEpeerreviewmaketitle

\section{Introduction}\label{sec:introduction}

As a tool for remote sensing and manipulation, a bilateral teleoperator strives to synchronize its master and slave robots tightly, and to provide its human operator with useful haptic cues about the slave-environment interactions. Therefore, the bilateral teleoperation feedback loop includes the human who operates the master, the environment which interacts with the slave and the master-communications-slave system~\cite{Chopra2013Automatica}. Because operators vary their dynamics according to their volition, and environments are typically unknown, neither are predictable or trvial to model~\cite{Hirche2012Proceedings}. Additionally, the master and slave exchange information distorted by time-varying communication delays~\cite{Lee2009TRO}. Thus, bilateral teleoperation is a nonlinear, time-varying and interconnected system with communication delays and uncertain user and environment dynamics~\cite{Jiang2010Auto}.

The physical interactions between the robotic master-communications-slave system~(aka the teleoperator) and the user and the environment~(aka its external terminators) involve exchange of energy~\cite{Stramigioli2007}. As a key theory for modeling and controlling the exchange of energy among interconnected systems, passivity is ofen pivotal to the rigorous treatment of closed-loop teleoperation without user and environment models~\cite{Hokayem2006}. Existing research exploits the stability of the feedback interconnection of passive systems by assuming passive operator and environment and by employing Lyapunov-like analysis~\cite{Nuno2011} or energy monitoring~\cite{FrankenStramigiogli2011} to offer controllers which maintain teleoperators with time delays passive.

Scattering-based, damping injection and adaptive strategies can all be designed to provably stabilize bilateral teleoperation with a unified Lyapunov-like energy function. Scattering or wave-based control can render the time-delayed communication channel passive, as well as reduce wave reflections~\cite{Bate2011} and improve trajectory tracking~\cite{Ye2010,Sun2018} and transparency~\cite{Niemeyer2006, Yalcin-Ohnishi2010, Rebelo2015, Du2016, Sun2018}. Damping injection control, Proportional-Derivative plus damping~(PD+d)~\cite{Lee2006} or Proportional plus damping~(P+d)~\cite{Nuno2008}, and extensions to position-force architectures with and without gravity compensation~\cite{Hua2010,Hua2012,Hua2013}, output feedback~\cite{Hua2011} and bounded actuation~\cite{Hashemzadeh2013,Hua2015,Zhai2016} implement a virtual spring between, and local dampers at, the master and slave sites. Joint-space~\cite{Chopra2008,Nuno2010} and task-space~\cite{Liu2013,Liu2015} adaptive strategies can synchronize master and slave robots with uncertain parameters and constant delays.

Energy monitoring-based control can render the teleoperator passive by dynamic damping injection, dynamic modulation of the nominal control force, or both. Time-domain passivity control~\cite{Ryu2004,Pan2013} and extensions to eliminate position drift~\cite{Chawda2015} inject sufficient damping to dissipate the delay-induced energy at each step. The energy bounding~\cite{Ryu2010} and passive set-position modulation~(PSPM)~\cite{Lee2010} strategies regulate the nominal control inputs to ensure that the teleoperator generates less energy than its physical and control damping dissipates. The two-layer approach~\cite{FrankenStramigiogli2011,Heck2018} modulates the forces computed in the transparency layer and adds damping in the passivity layer to limit energy accumulation in the system.

Input-output stability provides another path to rigorous stability of time-delay systems with uncertain dynamics~\cite{Gu2003}. In particular, input-to-state stability has diminished the conservatism of passive strategies in haptic rendering~\cite{Ryu2017TRO}, and has offered robust position tracking for time-delayed bilateral shared control of an aerial vehicle~\cite{LiuElSaddik2019}.

This paper introduces a novel dynamic interconnection and damping injection strategy for robust position tracking in time-delayed bilateral teleoperation based on a position-position structure. While force-reflection~\cite{Lawrence1993,Hashtrudi-Zaad2002,FrankenStramigiogli2011, Heck2018} and force-reproduction~\cite{NozakiOhnishi2014,Yokokura2014, Antonello2016, Nagatsu2017} architectures can improve transparency, force and acceleration measurements are unavailable or noisy for many commercial robots. Therefore, the proposed controllers employ only the position and velocity of the local robot and the delayed position of the remote robot. Rigorous Lyapunov stability analysis proves that they render the teleoperator input-to-state stable~(ISS).

The P+d~\cite{Nuno2008}, PSPM~\cite{Lee2010} and two-layer~\cite{FrankenStramigiogli2011,Franken2012} approaches are most closely related to the dynamic strategy introduced in this paper. Whereas the P+d and PSPM algorithms synchronize the master and slave exactly in the absence of user and environment forces but over-inject constant damping compared to the two-layer method, the dynamic strategy proposed in this paper offers a unique property for robust position tracking in time-delayed bilateral teleoperation: given full~(unlimited) actuation, it can confine the master-slave position error to a prescribed invariant set, and can drive it exponentially to a globally attractive set which includes the origin. More importantly, the control gains determine the Lebesgue measure of the invariant and globally attractive sets, and the rate of convergece to the latter. Therefore, sufficiently large control gains can tighten the master-slave coupling, and can indirectly convey the slave-environment interactions to the operator~\cite{Heck2018}. The unique property of the dynamic interconnection and damping injection strategy is advantageous for high-precision telemanipulation tasks with position tracking error constraints. 
%

\section{Preliminaries}\label{sec: preliminaries}

\subsection{System Dynamics}

Let the master and slave robots be $n$-degree-of-freedom~($n$-DOF) serial manipulators with revolute joints. Their joint space dynamics are:
\begin{equation}\label{equ1}
\begin{aligned}
\fbm{M}_{m}(\fbm{q}_{m})\cdot\ddotbm{q}_{m}+\fbm{C}_{m}(\fbm{q}_{m},\dotbm{q}_{m})\cdot\dotbm{q}_{m}=&\bm{\tau}_{m}+\bm{\tau}_{h}\textrm{,}\\
\fbm{M}_{s}(\fbm{q}_{s})\cdot\ddotbm{q}_{s}+\fbm{C}_{s}(\fbm{q}_{s},\dotbm{q}_{s})\cdot\dotbm{q}_{s}=&\bm{\tau}_{s}+\bm{\tau}_{e}\textrm{,}
\end{aligned}
\end{equation} 
where the subscript $i=m,s$ indexes master and slave quantities, $\ddotbm{q}_{i}$, $\dotbm{q}_{i}$ and $\fbm{q}_{i}$ are joint acceleration, velocity and position, $\fbm{M}_{i}(\fbm{q}_{i})$ and $\fbm{C}_{i}(\fbm{q}_{i},\dotbm{q}_{i})$ are matrices of inertia and of Coriolis and centrifugal effects, $\bm{\tau}_{i}$ are control torques, and $\bm{\tau}_{h}$ and $\bm{\tau}_{e}$ are user and environment torques.

The following properties of the dynamics~\eqref{equ1}, and assumptions on communication delays and on user and environment torques, facilitate later control design and stability analysis.
\begin{enumerate}[label=P.\arabic*]
\item \label{item:P1}
The inertia matrix $\fbm{M}_{i}(\fbm{q}_{i})$ is symmetric, positive definite and uniformly bounded by $\fbm{0}\prec\lambda_{i1}\fbm{I}\preceq \fbm{M}_{i}(\fbm{q}_{i})\preceq \lambda_{i2}\fbm{I}\prec \bm{\infty}$, with $\lambda_{i1}$ and $\lambda_{i2}$ positive constants.
\item \label{item:P2}
The matrix $\dot{\fbm{M}}_{i}(\fbm{q}_{i})-2\fbm{C}_{i}(\fbm{q}_{i},\dot{\fbm{q}}_{i})$ is skew-symmetric.
\item \label{item:P3}
There exists $c_{i}>0$ such that $\|\fbm{C}_{i}(\fbm{q}_{i},\fbm{x})\cdot\fbm{y}\|\leq c_{i}\cdot\|\fbm{x}\|\cdot \|\fbm{y}\|, \forall \fbm{q}_{i}, \fbm{x}, \fbm{y}$.
\end{enumerate} 
\begin{enumerate}[label=A.\arabic*]
\item \label{item:A1}
The time-varying communication delays from robot $i$ to robot $j$, $d_{i}$, are bounded, $0\leq d_{i}\leq\overline{d}_{i}$, for $i,j=m,s$.
\item \label{item:A2}
The joint torques due to operator $\bm{\tau}_{h}$ and environment $\bm{\tau}_{e}$ forces are bounded by $\|\bm{\tau}_{k}\|\leq\overline{\tau}_{k}$, $k=h,e$.
\end{enumerate}

\subsection{Input-to-State Stability}\label{sec: iss}

This section overviews the definitions and theorems required by the stability analysis in~\sect{sec: main result}.

A function $\alpha: \mathbb{R}_{\geq 0}\mapsto\mathbb{R}_{\geq 0}$ is of class $\mathcal{K}$ if it is continuous, strictly increasing and $\alpha(0)=0$; of class $\mathcal{K}_{\infty}$ if it is of class $\mathcal{K}$ and unbounded; of class $\mathcal{L}$ if it decreases to zero as its argumet tends to $+\infty$. A function $\beta: \mathbb{R}_{\geq 0}\times\mathbb{R}_{\geq 0}\mapsto\mathbb{R}_{\geq 0}$ is of class $\mathcal{KL}$ if it is of class $\mathcal{K}$ in its first argument and of class $\mathcal{L}$ in the second argument. Let $\mathcal{C}([-r,0];\mathbb{R}^{m})$ denote the set of the continuous functions defined on $[-r,0]$ and with values in $\mathbb{R}^{m}$. For any essentially bounded function $\bm{\phi}\in \mathcal{C}([-r,0];\mathbb{R}^{m})$, let $|\bm{\phi}|_{r}=\sup\limits_{-r\leq\tau\leq 0}\|\bm{\phi}(\tau)\|$ and $|\bm{\phi}|_{a}$ be a norm of $\bm{\phi}$ such that:
\begin{equation}\label{equ2}
\gamma_{a}\cdot \|\bm{\phi}(0)\|\leq |\bm{\phi}|_{a}\leq\overline{\gamma}_{a}\cdot |\bm{\phi}|_{r}
\end{equation}
for some positive reals $\gamma_{a}$ and $\overline{\gamma}_{a}$.

\begin{enumerate}[label=D.\arabic*]
\item \label{item:D1}~\cite{Sontag2008}
The nonlinear delay-free system
\begin{equation}\label{equ3}
\dotbm{x}(t)=f(\fbm{x}(t),\fbm{u}(t))
\end{equation}
is ISS with input $\fbm{u}(t)\in\mathbb{R}^{l}$ and state $\fbm{x}(t)\in\mathbb{R}^{m}$ if there exist functions $\alpha\in\mathcal{K}$ and $\beta\in\mathcal{KL}$ such that:
\begin{equation}\label{equ4}
\|\fbm{x}(t)\|\leq \beta\left(\|\fbm{x}(0)\|,t\right)+\alpha\left(\sup\limits_{0\leq\tau\leq t}\|\fbm{u}(\tau)\|\right)\textrm{,}\quad \forall t\geq 0\textrm{.}
\end{equation}
\item \label{item:D2}~\cite{Jiang2006}
The nonlinear time-delay system:
\begin{equation}\label{equ5}
\begin{aligned}
\dotbm{x}(t)=& f(\fbm{x}_{t},\fbm{u}(t))\textrm{,}\quad t\geq 0\ a.e.\textrm{,}\\
\fbm{x}(\tau)=&\bm{\xi}_{0}(\tau)\textrm{,}\quad \tau\in [-r,0]\textrm{,}
\end{aligned}
\end{equation}
with $\fbm{x}_{t}:[-r,0]\mapsto\mathbb{R}^{m}$ the standard function $\fbm{x}_{t}(\tau)=\fbm{x}(t+\tau)$, and $r$ the maximum involved delay, $f: \mathcal{C}([-r,0];\mathbb{R}^{m})\times\mathbb{R}^{l}\mapsto\mathbb{R}^{m}$, and $\bm{\xi}_{0}\in \mathcal{C}([-r,0];\mathbb{R}^{m})$, is ISS with input $\fbm{u}(t)\in\mathbb{R}^{l}$ and state $\fbm{x}(t)\in\mathbb{R}^{m}$ if there exist functions $\alpha\in\mathcal{K}$ and $\beta\in\mathcal{KL}$ such that:
\begin{equation}\label{equ6}
\|\fbm{x}(t)\|\leq\beta\left(|\bm{\xi}_{0}|_{r},t\right)+\alpha\left(\sup\limits_{0\leq\tau\leq t}\|\fbm{u}(\tau)\|\right)\textrm{,}\quad \forall t\geq 0\textrm{.}
\end{equation}
\end{enumerate}

Because the time-delay system~\eqref{equ5} is infinite dimensional~\cite{Gu2003}, the input-to-state stability of it is defined differently than that of the delay-free system~\eqref{equ3}.  Correspondingly, the following two theorems using ISS-Lyapunov functions and Lyapunov-Krasovskii functionals facilitate proving ISS teleoperation without delays and with time-varying delays, respectively.

\begin{enumerate}[label=T.\arabic*]
\item \label{item:T1}~\cite{Sontag2008}
The delay-free system~\eqref{equ3} is ISS if and only if there exist an ISS-Lyapunov function $V:\mathbb{R}^{m}\mapsto\mathbb{R}_{\geq 0}$, and functions $\alpha_{1},\alpha_{2}\in\mathcal{K}_{\infty}$, $\alpha_{3},\rho\in\mathcal{K}$ such that:
\begin{enumerate}
\item 
$\alpha_{1}(\|\fbm{x}\|)\leq V(\fbm{x})\leq \alpha_{2}(\|\fbm{x}\|)$, $\forall\fbm{x}\in\mathbb{R}^{m}$;
\item 
$\dot{V}(\fbm{x},\fbm{u})\leq -\alpha_{3}(\|\fbm{x}\|)$, $\forall \fbm{x}\in\mathbb{R}^{m}, \fbm{u}\in\mathbb{R}^{l}: \|\fbm{x}\|\geq\rho(\|\fbm{u}\|)$.
\end{enumerate}
\item \label{item:T2}~\cite{Jiang2006}
The time-delay system~\eqref{equ5} is ISS if there is a Lyapunov-Krasovskii functional $V:\mathcal{C}([-r,0];\mathbb{R}^{m})\mapsto\mathbb{R}_{\geq 0}$, functions $\alpha_{1},\alpha_{2}\in\mathcal{K}_{\infty}$, $\alpha_{3},\rho\in\mathcal{K}$ such that:
\begin{enumerate}
\item 
$\alpha_{1}(\|\fbm{x}\|)\leq V(\fbm{x}_{t})\leq \alpha_{2}(|\fbm{x}_{t}|_{a})$, $\forall\fbm{x}_{t}\in \mathcal{C}([-r,0];\mathbb{R}^{m})$;
\item 
$\dot{V}(\fbm{x}_{t},\fbm{u})\leq -\alpha_{3}(|\fbm{x}_{t}|_{a})$, $\forall \fbm{x}_{t}\in \mathcal{C}([-r,0];\mathbb{R}^{m}), \fbm{u}\in\mathbb{R}^{l}: |\fbm{x}_{t}|_{a}\geq\rho(\|\fbm{u}\|)$.
\end{enumerate}
\end{enumerate}

%

\section{Main Result}\label{sec: main result}

This section presents the control design and stability analysis for ISS bilateral teleoperation, considering communications both without, and with time-varying, delays.

\subsection{ISS Teleoperation Without Time Delays}\label{sec: no delay}

Define sliding surfaces for the master and slave robots by:
\begin{equation}\label{equ7}
\begin{aligned}
\fbm{s}_{m}=& \dotbm{q}_{m}+\sigma\cdot(\fbm{q}_{m}-\fbm{q}_{s})\textrm{,}\\
\fbm{s}_{s}=& \dotbm{q}_{s}+\sigma\cdot(\fbm{q}_{s}-\fbm{q}_{m})\textrm{,}
\end{aligned}
\end{equation}
where $\sigma>0$ is a constant to be determined. Then, the dynamics~\eqref{equ1} can be transformed to:
\begin{equation}\label{equ8}
\begin{aligned}
\fbm{M}_{m}(\fbm{q}_{m})\cdot\dotbm{s}_{m}+\fbm{C}_{m}(\fbm{q}_{m},\dotbm{q}_{m})\cdot\fbm{s}_{m}=&\bm{\tau}_{m}+\bm{\tau}_{h}+\sigma\cdot\fbm{\Delta}_{m}\textrm{,}\\
\fbm{M}_{s}(\fbm{q}_{s})\cdot\dotbm{s}_{s}+\fbm{C}_{s}(\fbm{q}_{s},\dotbm{q}_{s})\cdot\fbm{s}_{s}=&\bm{\tau}_{s}+\bm{\tau}_{e}+\sigma\cdot\fbm{\Delta}_{s}\textrm{,}
\end{aligned}
\end{equation}
where the state-dependent master and slave mismatches are: 
\begin{equation}\label{equ9}
\begin{aligned}
\fbm{\Delta}_{m}=&\fbm{M}_{m}(\fbm{q}_{m})\cdot(\dotbm{q}_{m}-\dotbm{q}_{s})+\fbm{C}_{m}(\fbm{q}_{m},\dotbm{q}_{m})\cdot(\fbm{q}_{m}-\fbm{q}_{s})\textrm{,}\\
\fbm{\Delta}_{s}=&\fbm{M}_{s}(\fbm{q}_{s})\cdot(\dotbm{q}_{s}-\dotbm{q}_{m})+\fbm{C}_{s}(\fbm{q}_{s},\dotbm{q}_{s})\cdot(\fbm{q}_{s}-\fbm{q}_{m})\textrm{.}
\end{aligned}
\end{equation}

The master and slave receive each other's position instantly, and their dynamic interconnection and damping injection controllers are:
\begin{equation}\label{equ10}
\begin{aligned}
\bm{\tau}_{m}=&-\fbm{K}_{m}(\dotbm{q}_{m})\cdot\fbm{s}_{m}-\fbm{P}\cdot(\fbm{q}_{m}-\fbm{q}_{s})-\fbm{D}_{m}\dotbm{q}_{m}\textrm{,}\\
\bm{\tau}_{s}=&-\fbm{K}_{s}(\dotbm{q}_{s})\cdot\fbm{s}_{s}-\fbm{P}\cdot(\fbm{q}_{s}-\fbm{q}_{m})-\fbm{D}_{s}\dotbm{q}_{s}\textrm{,}
\end{aligned}
\end{equation}
where $\fbm{K}_{i}(\dotbm{q}_{i})$, $\fbm{P}$ and $\fbm{D}_{i}$, $i=m,s$, are diagonal positive definite gain matrices to be designed. 

\begin{remark}\label{rem1}
\normalfont The velocity-dependent gains $\fbm{K}_{i}(\dotbm{q}_{i})$ are designed to suppress the state-dependent mismatch $\fbm{\Delta}_{i}$ in~\eqref{equ9}. Rewritting~\eqref{equ10} in the form: 
\begin{align*}
\bm{\tau}_{i}=-\Big[\fbm{P}+\sigma\cdot\fbm{K}_{i}(\dotbm{q}_{i})\Big]\cdot(\fbm{q}_{i}-\fbm{q}_{j})-\Big[\fbm{D}_{i}+\fbm{K}_{i}(\dotbm{q}_{i})\Big]\cdot\dotbm{q}_{i}\textrm{,}
\end{align*}
where $i,j=m,s$ and $i\neq j$, reveals that the controller dynamically modulates the master-slave interconnection and the local damping through $\fbm{K}_{i}(\dotbm{q}_{i})$. The transformed system dynamics~\eqref{equ8} are input-output passive with input $\bm{\tau}_{i}+\bm{\tau}_{k}+\sigma\cdot\fbm{\Delta}_{i}$ and output $\fbm{s}_{i}$, where $i=m,s$ and $k=h,e$, respectively. However, the teleoperator~\eqref{equ1} in closed-loop with the controller~\eqref{equ10} is not guaranteed stable because the modulated interconnection may be non-passive~\cite{Lee2010} and thus may render the teleoperator non-passive. The following rigorous analysis is required to show ISS teleoperation.
\end{remark}


The following Lyapunov candidate:
\begin{equation}\label{equ11}
\begin{aligned}
V=\frac{1}{2}\sum_{i=m,s}\tbm{s}_{i}\fbm{M}_{i}(\fbm{q}_{i})\cdot\fbm{s}_{i}+\frac{1}{2}\cdot(\fbm{q}_{m}-\fbm{q}_{s})^\mathsf{T}\fbm{P}\cdot(\fbm{q}_{m}-\fbm{q}_{s})\textrm{,}
\end{aligned}
\end{equation} 
serves to evaluate stability. After using property~\ref{item:P2}, the time derivative of $V$ along the closed-loop dynamics~\eqref{equ8} with the control~\eqref{equ10} can be evaluated as:
\begin{equation}\label{equ12}
\begin{aligned}
\dot{V}=&\tbm{s}_{m}\bm{\tau}_{h}+\dottbm{q}_{m}\fbm{P}\cdot(\fbm{q}_{m}-\fbm{q}_{s})-\tbm{s}_{m}\fbm{K}_{m}(\dotbm{q}_{m})\cdot\fbm{s}_{m}\\
&-\tbm{s}_{m}\fbm{P}\cdot(\fbm{q}_{m}-\fbm{q}_{s})-\tbm{s}_{m}\fbm{D}_{m}\dotbm{q}_{m}+\sigma\cdot\tbm{s}_{m}\fbm{\Delta}_{m}\\
&+\tbm{s}_{s}\bm{\tau}_{e}+\dottbm{q}_{s}\fbm{P}\cdot(\fbm{q}_{s}-\fbm{q}_{m})-\tbm{s}_{s}\fbm{K}_{s}(\dotbm{q}_{s})\cdot\fbm{s}_{s}\\
&-\tbm{s}_{s}\fbm{P}\cdot(\fbm{q}_{s}-\fbm{q}_{m})-\tbm{s}_{s}\fbm{D}_{s}\dotbm{q}_{s}+\sigma\cdot\tbm{s}_{s}\fbm{\Delta}_{s}\textrm{.}
\end{aligned}
\end{equation}

The definitions of $\fbm{\Delta}_{i}$ in~\eqref{equ9} together with properties~\ref{item:P1} and~\ref{item:P3} lead to the following inequalities for $i=m,s$:
\begin{equation}\label{equ13}
\begin{aligned}
\tbm{s}_{i}\fbm{\Delta}_{i}\leq&\Big|\tbm{s}_{i}\fbm{M}_{i}(\fbm{q}_{i})\cdot(\dotbm{q}_{m}-\dotbm{q}_{s})\Big|\\
&\quad \quad \quad \ \ +\Big|\tbm{s}_{i}\fbm{C}_{i}(\fbm{q}_{i},\dotbm{q}_{i})\cdot(\fbm{q}_{m}-\fbm{q}_{s})\Big|\\
\leq &\lambda_{i2}\cdot\left(\tbm{s}_{i}\fbm{s}_{i}+\frac{1}{2}\dottbm{q}_{m}\dotbm{q}_{m}+\frac{1}{2}\dottbm{q}_{s}\dotbm{q}_{s}\right)\\
&+c_{i}\cdot\left[\|\dotbm{q}_{i}\|^{2}\cdot\tbm{s}_{i}\fbm{s}_{i}+\frac{1}{4}\cdot(\fbm{q}_{m}-\fbm{q}_{s})^\mathsf{T}(\fbm{q}_{m}-\fbm{q}_{s})\right]\textrm{.}
\end{aligned}
\end{equation} 

\begin{remark}\label{rem2}
\normalfont The inequalities~\eqref{equ13} show that the impact of master and slave mismatches can be upper-bounded using the velocities of the two robots and their position error. Thus, they indicate the demand for a dynamic interconnection and damping injection strategy. More specifically, as will be illustrated in~\eqref{equ19}, the terms $\|\dotbm{q}_{i}\|^{2}\cdot\tbm{s}_{i}\fbm{s}_{i}$ in inequalities~\eqref{equ13} can be dominated by selecting velocity-dependent gains $\fbm{K}_{i}(\dotbm{q}_{i})$ in the controls $\bm{\tau}_{i}$. Alternatively, the terms $\Big|\tbm{s}_{i}\fbm{C}_{i}(\fbm{q}_{i},\dotbm{q}_{i})\cdot(\fbm{q}_{m}-\fbm{q}_{s})\Big|$ can be bounded by $c_{i}\cdot\left[\|\fbm{q}_{m}-\fbm{q}_{s}\|^{2}\cdot\tbm{s}_{i}\fbm{s}_{i}+\frac{1}{4}\cdot\dottbm{q}_{i}\dotbm{q}_{i}\right]$, and the gains $\fbm{K}_{i}$ can be updated based on the master-slave position error. Therefore, both the master and slave velocities and their position error can be used to dynamically modulate their coupling and local damping injection, to achieve ISS bilateral teleoperation.
\end{remark}

The definitions of the sliding surfaces $\fbm{s}_{i}$, $i=m,s$ lead to:
\begin{equation}\label{equ14}
\begin{aligned}
&-\tbm{s}_{m}\fbm{P}\cdot(\fbm{q}_{m}-\fbm{q}_{s})-\tbm{s}_{m}\fbm{D}_{m}\dotbm{q}_{m}+\dottbm{q}_{m}\fbm{P}\cdot(\fbm{q}_{m}-\fbm{q}_{s})\\
=&-\sigma\cdot(\fbm{q}_{m}-\fbm{q}_{s})^\mathsf{T}\fbm{P}\cdot(\fbm{q}_{m}-\fbm{q}_{s})-\dottbm{q}_{m}\fbm{D}_{m}\dotbm{q}_{m}\\
&-\sigma\cdot(\fbm{q}_{m}-\fbm{q}_{s})^\mathsf{T}\fbm{D}_{m}\dotbm{q}_{m}\\
\leq &-\sigma\cdot(\fbm{q}_{m}-\fbm{q}_{s})^\mathsf{T}\left(\fbm{P}-\frac{\mu_{m}}{4}\cdot\fbm{D}_{m}\right)\cdot(\fbm{q}_{m}-\fbm{q}_{s})\\
&-\left(1-\frac{\sigma}{\mu_{m}}\right)\cdot\dottbm{q}_{m}\fbm{D}_{m}\dotbm{q}_{m}\textrm{,}
\end{aligned}
\end{equation}
where $\mu_{m}>0$, and, similarly, to:
\begin{equation}\label{equ15}
\begin{aligned}
&-\tbm{s}_{s}\fbm{P}\cdot(\fbm{q}_{s}-\fbm{q}_{m})-\tbm{s}_{s}\fbm{D}_{s}\dotbm{q}_{s}+\dottbm{q}_{s}\fbm{P}\cdot(\fbm{q}_{s}-\fbm{q}_{m})\\
\leq &-\sigma\cdot(\fbm{q}_{s}-\fbm{q}_{m})^\mathsf{T}\left(\fbm{P}-\frac{\mu_{s}}{4}\cdot\fbm{D}_{s}\right)\cdot(\fbm{q}_{s}-\fbm{q}_{m})\\
&-\left(1-\frac{\sigma}{\mu_{s}}\right)\cdot\dottbm{q}_{s}\fbm{D}_{s}\dotbm{q}_{s}\textrm{,}
\end{aligned}
\end{equation}
where $\mu_{s}>0$. Further algebraic manipulations yield:
\begin{equation}\label{equ16}
\tbm{s}_{m}\bm{\tau}_{h}+\tbm{s}_{s}\bm{\tau}_{e}\leq\sum_{i=m,s}\omega_{i}\cdot\tbm{s}_{i}\fbm{s}_{i}+\frac{1}{4}\cdot\left(\frac{\|\bm{\tau}_{h}\|^{2}}{\omega_{m}}+\frac{\|\bm{\tau}_{e}\|^{2}}{\omega_{s}}\right)\textrm{,}
\end{equation}
where $\omega_{i}>0$, $i=m,s$.

After substitutions from~\eqs{equ13}{equ16}, the derivative of the Lyapunov candidate~\eqref{equ12} can then be bounded by:
\begin{equation}\label{equ17}
\begin{aligned}
\dot{V}\leq &-\sum_{i=m,s}\Big(\tbm{s}_{i}\overline{\fbm{K}}_{i}(\dotbm{q}_{i})\cdot\fbm{s}_{i}+\dottbm{q}_{i}\overline{\fbm{D}}_{i}\dotbm{q}_{i}\Big)\\
&-(\fbm{q}_{m}-\fbm{q}_{s})^\mathsf{T}\overline{\fbm{P}}\cdot(\fbm{q}_{m}-\fbm{q}_{s})+\frac{1}{4}\cdot\left(\frac{\|\bm{\tau}_{h}\|^{2}}{\omega_{m}}+\frac{\|\bm{\tau}_{e}\|^{2}}{\omega_{s}}\right)\textrm{,}
\end{aligned}
\end{equation}
where:  
\begin{equation}\label{equ18}
\begin{aligned}
\overline{\fbm{K}}_{i}(\dotbm{q}_{i})=&\fbm{K}_{i}(\dotbm{q}_{i})-\Big(\omega_{i}+\sigma\cdot\lambda_{i2}+\sigma\cdot c_{i}\cdot\|\dotbm{q}_{i}\|^{2}\Big)\cdot \fbm{I}\textrm{,}\\
\overline{\fbm{D}}_{i}=&\left(1-\frac{\sigma}{\mu_{i}}\right)\cdot\fbm{D}_{i}-\frac{\sigma}{2}\cdot\big(\lambda_{m2}+\lambda_{s2}\big)\cdot\fbm{I}\textrm{,}\\
\overline{\fbm{P}}=&2\sigma\cdot\fbm{P}-\frac{\sigma}{4}\cdot\sum_{i=m,s}\big(\mu_{i}\cdot\fbm{D}_{i}+c_{i}\cdot\fbm{I}\big)\textrm{.}
\end{aligned}
\end{equation}

\begin{theorem}\label{theorem1}
The teleoperator~\eqref{equ1} with the control~\eqref{equ10} is ISS with input $\fbm{u}=[\tfbm{\tau}_{h}\ \tfbm{\tau}_{e}]^\mathsf{T}$ and state $\fbm{x}=[\dottbm{q}_{m}\ \dottbm{q}_{s}\ (\fbm{q}_{m}-\fbm{q}_{s})^\mathsf{T}]^\mathsf{T}$ if the control gains $\fbm{K}_{i}(\dotbm{q}_{i})$, $\fbm{D}_{i}$, $\fbm{P}$, $\sigma$ and positive parameters $\mu_{i}$, $\omega_{i}$, $\kappa$, $i=m,s$, satisfy:
\begin{equation}\label{equ19}
\begin{aligned}
\overline{\fbm{P}}\succeq \frac{\kappa}{2}\cdot\fbm{P}\textrm{,}\quad \overline{\fbm{K}}_{i}(\dotbm{q}_{i})\succeq \frac{\kappa}{2}\cdot\lambda_{i2}\cdot\fbm{I}\quad \text{and}\quad \overline{\fbm{D}}_{i}\succeq\fbm{0}\textrm{.}
\end{aligned}
\end{equation}
\end{theorem}

\begin{proof}

The definitions of the sliding surfaces $\fbm{s}_{i}$ imply that:
\begin{equation}\label{equ20}
\begin{aligned}
\|\dotbm{q}_{i}\|^{2}\leq & 2\cdot\|\fbm{s}_{i}\|^{2}+2\sigma^{2}\cdot\|\fbm{q}_{m}-\fbm{q}_{s}\|^{2}\textrm{,}\\
\|\fbm{s}_{i}\|^{2}\leq & 2\cdot\|\dotbm{q}_{i}\|^{2}+2\sigma^{2}\cdot\|\fbm{q}_{m}-\fbm{q}_{s}\|^{2}\textrm{.}
\end{aligned}
\end{equation}
Then, Property~\ref{item:P1} and Equations~\eqref{equ11} and~\eqref{equ20} lead to:
\begin{equation}\label{equ21}
\begin{aligned}
V\geq \sum_{i=m,s}\frac{\lambda_{i1}}{2}\cdot\|\fbm{s}_{i}\|^{2}+\frac{p}{2}\cdot\|\fbm{q}_{m}-\fbm{q}_{s}\|^{2}\geq \alpha_{1}(\|\fbm{x}\|)\textrm{,}
\end{aligned}
\end{equation}
for $p$ the minimum eigenvalue of $\fbm{P}$, and $\alpha_{1}(\|\fbm{x}\|)=a_{1}\cdot\|\fbm{x}\|^{2}$ with:
\begin{align*}
\frac{1}{a_{1}}=\frac{4}{\lambda_{m1}}+\frac{4}{\lambda_{s1}}+\frac{8\sigma^{2}+2}{p}\textrm{,}
\end{align*}
and also to
\begin{equation}\label{equ22}
\begin{aligned}
V\leq &\sum_{i=m,s}\frac{\lambda_{i2}}{2}\cdot\|\fbm{s}_{i}\|^{2}+\frac{P}{2}\cdot\|\fbm{q}_{m}-\fbm{q}_{s}\|^{2}\leq \alpha_{2}(\|\fbm{x}\|)\textrm{,}
\end{aligned}
\end{equation}
for $P$ the maximum eigenvalue of $\fbm{P}$, and $\alpha_{2}(\|\fbm{x}\|)=a_{2}\cdot\|\fbm{x}\|^{2}$ with:
\begin{align*}
a_{2}=\max\left(\lambda_{m2},\lambda_{s2},\frac{P+2(\lambda_{m2}+\lambda_{s2})\cdot\sigma^{2}}{2}\right)\textrm{.}
\end{align*}
Because functions $\alpha_{1}$ and $\alpha_{2}$ are of class $\mathcal{K}_{\infty}$, $V$ satisfies condition a) of Theorem~\ref{item:T1}.

If condition~\eqref{equ19} is satisfied, $\dot{V}$ in~\eqref{equ17} can be further upper-bounded by:
\begin{equation}\label{equ23}
\dot{V}\leq -\kappa\cdot V+\frac{1}{4\omega}\cdot\|\fbm{u}\|^{2}\textrm{,}
\end{equation}
where $\omega=\min(\omega_{m},\omega_{s})$. Then, choosing class $\mathcal{K}$ functions:
\begin{equation}\label{equ24}
\alpha_{3}(\|\fbm{x}\|)= \frac{a_{1}\kappa}{2}\cdot \|\fbm{x}\|^{2}\textrm{,}\ \text{and}\ \rho(\|\fbm{u}\|) =  \sqrt{\frac{1}{2a_{1}\kappa\omega}}\cdot\|\fbm{u}\|\textrm{,} 
\end{equation}
ensures that $V$ in~\eqref{equ11} also satisfies condition b) of Theorem~\ref{item:T1} and thus, that the teleoperator is ISS.
\end{proof}

\begin{corollary}\label{corollary1}
ISS teleoperation under the control~\eqref{equ10} renders invariant the set:
\begin{equation}\label{equ25}
\mathcal{S}_{I}=\left\{\fbm{q}_{m}-\fbm{q}_{s}:\ \|\fbm{q}_{m}-\fbm{q}_{s}\|^{2}\leq \frac{2}{p}\cdot\left(V_{0}+\frac{\overline{\tau}^{2}}{4\kappa\omega}\right)\right\}\textrm{,}
\end{equation}
and globally exponentially attractive the set:
\begin{equation}\label{equ26}
\mathcal{S}_{A}=\left\{\fbm{q}_{m}-\fbm{q}_{s}:\ \|\fbm{q}_{m}-\fbm{q}_{s}\|^{2}\leq \frac{\overline{\tau}^{2}}{2p\kappa\omega}\right\}\textrm{,}
\end{equation}
where $\overline{\tau}^{2}=\bar{\tau}^{2}_{h}+\bar{\tau}^{2}_{e}$.
\end{corollary}

\begin{proof}
Time integration of $\dot{V}$ in~\eqref{equ23} yields:
\begin{equation}\label{equ27}
\begin{aligned}
V(t)= & e^{-\kappa t}\cdot V(0)+\frac{1}{4\omega}\int^{t}_{0}e^{-\kappa (t-\theta)}\cdot\|\fbm{u}(\theta)\|^{2}d\theta\\
\leq & e^{-\kappa t}\cdot V(0)+\frac{\overline{\tau}^{2}}{4\omega}\int^{t}_{0}e^{-\kappa (t-\theta)}d\theta\\
\leq & e^{-\kappa t}\cdot V(0)+\frac{\overline{\tau}^{2}}{4\kappa\omega}\textrm{,}
\end{aligned}
\end{equation}
which, together with~\eqref{equ21}, complete the proof.
\end{proof}

\subsection{ISS Teleoperation With Time-Varying Delays}\label{sec: delay}

For each robot, construct the following auxiliary system~(proxy):
\begin{equation}\label{equ28}
\begin{aligned}
\hbm{M}_{i}\ddothbm{q}_{i}=&-\hbm{K}_{i}\cdot\left[\dothatbm{q}_{i}+\hat{\sigma}\cdot\left(\fbm{P}_{i}\cdot(\hbm{q}_{i}-\fbm{q}_{i})+\hbm{P}\cdot(\hbm{q}_{i}-\hbm{q}_{jd})\right)\right]\\
&-\hbm{D}_{i}\dothatbm{q}_{i}-\fbm{P}_{i}\cdot(\hbm{q}_{i}-\fbm{q}_{i})-\hbm{P}\cdot(\hbm{q}_{i}-\hbm{q}_{jd})\textrm{,}
\end{aligned}
\end{equation} 
where: $i,j=m,s$ and $i\neq j$; $\hat{\sigma}>0$; $\hbm{M}_{i}$, $\hbm{K}_{i}$, $\hbm{D}_{i}$, $\fbm{P}_{i}$ and $\hbm{P}$ are diagonal positive definite matrices to be determined; and $\hbm{q}_{jd}=\hbm{q}_{j}(t-d_{j})$ is the output of the auxiliary system of robot $j$ received with a delay $d_{j}$ by robot $i$. Let the sliding surface of the proxy of robot $i$ be:
\begin{equation}\label{equ29}
\hbm{s}_{i}=\dothatbm{q}_{i}+\hat{\sigma}\cdot\hbm{e}_{i}\textrm{,}
\end{equation}
with:
\begin{equation}\label{equ30}
\hbm{e}_{i}=\fbm{P}_{i}\cdot\left(\hbm{q}_{i}-\fbm{q}_{i}\right)+\hbm{P}\cdot\left(\hbm{q}_{i}-\hbm{q}_{j}\right)\textrm{.}
\end{equation}
Then, the auxiliary dynamics can be rearranged in the form:
\begin{equation}\label{equ31}
\begin{aligned}
\hbm{M}_{i}\dothatbm{s}_{i}=&\hat{\sigma}\cdot\hbm{M}_{i}\left[\fbm{P}_{i}\cdot\left(\dothatbm{q}_{i}-\dotbm{q}_{i}\right)+\hbm{P}\cdot\left(\dothatbm{q}_{i}-\dothatbm{q}_{j}\right)\right]-\hbm{e}_{i}\\
&-\hbm{K}_{i}\hbm{s}_{i}-\left(\hat{\sigma}\cdot\hbm{K}_{i}+\fbm{I}\right)\cdot\hbm{P}\cdot\left(\hbm{q}_{j}-\hbm{q}_{jd}\right)-\hbm{D}_{i}\dothatbm{q}_{i}\textrm{.}
\end{aligned}
\end{equation}

Let the sliding surface of robot $i$ be:
\begin{equation}\label{equ32}
\fbm{s}_{i}=\dotbm{q}_{i}+\sigma\cdot\left(\fbm{q}_{i}-\hbm{q}_{i}\right)\textrm{,}
\end{equation}
with $\sigma>0$. As in~\sect{sec: no delay}, the master and slave dynamics can be transformed into~\eqref{equ8} but with mismatch:
\begin{equation}\label{equ33}
\begin{aligned}
\fbm{\Delta}_{i}=&\fbm{M}_{i}(\fbm{q}_{i})\cdot\left(\dotbm{q}_{i}-\dothatbm{q}_{i}\right)+\fbm{C}_{i}(\fbm{q}_{i},\dotbm{q}_{i})\cdot\left(\fbm{q}_{i}-\hbm{q}_{i}\right)\textrm{.}
\end{aligned}
\end{equation}
Correspondingly, the master and slave controllers are designed by:
\begin{equation}\label{equ34}
\begin{aligned}
\bm{\tau}_{i}=-\fbm{K}_{i}(\dotbm{q}_{i})\cdot\fbm{s}_{i}-\fbm{P}_{i}\cdot(\fbm{q}_{i}-\hbm{q}_{i})-\fbm{D}_{i}\dotbm{q}_{i}\textrm{,}
\end{aligned}
\end{equation}
with $\fbm{K}_{i}(\dotbm{q}_{i})$ and $\fbm{D}_{i}$ diagonal positive definite gain matrices to be determined.


\begin{remark}\label{rem3}
\normalfont The auxiliary systems~\eqref{equ28} have inertia $\hbm{M}_{i}$, are connected to each other through static Proportional control, and are driven by the master and slave through static interconnection and damping control. In contrast, the master and slave are connected to their proxies by the dynamic interconnection and damping injection controls~\eqref{equ34}, which are updated according to each robot velocity $\dotbm{q}_{i}$. Thus, the delays distort only information transmitted between statically coupled proxies and classical damping injection~\cite{Nuno2008} can overcome their destabilizing effect. The state-dependent mismatches $\bm{\Delta}_{i}$ in~\eqref{equ33} affect only the master and slave, as in the case of non-delayed communications in~\sect{sec: no delay}. The dynamic control strategy~\eqref{equ34} will be designed in~\eqref{equ37} to address these mismatches. 
\end{remark}

Stability is validated using the Lyapunov candidate $V=V_{1}+V_{2}$ with:
\begin{equation}\label{equ35}
\begin{aligned}
V_{1}=&\frac{1}{2}\sum_{i=m,s}\left[\tbm{s}_{i}\fbm{M}_{i}(\fbm{q}_{i})\cdot\fbm{s}_{i}+\left(\fbm{q}_{i}-\hbm{q}_{i}\right)^\mathsf{T}\fbm{P}_{i}\cdot\left(\fbm{q}_{i}-\hbm{q}_{i}\right)\right]\\
&+\frac{1}{2}\sum_{i=m,s}\thbm{s}_{i}\hbm{M}_{i}\hbm{s}_{i}+\frac{1}{2}\left(\hbm{q}_{m}-\hbm{q}_{s}\right)^\mathsf{T}\hbm{P}\cdot\left(\hbm{q}_{m}-\hbm{q}_{s}\right)\textrm{,}\\
V_{2}=&\sum_{i=m,s}\int^{0}_{-\overline{d}_{i}}\int^{t}_{t+\theta}e^{-\gamma(t-\xi)}\cdot\dothattbm{q}_{i}(\xi)\fbm{Q}_{i}\dothatbm{q}_{i}(\xi)d\xi d\theta\textrm{,}
\end{aligned}
\end{equation}
where $\fbm{Q}_{i}\succ\fbm{0}$, $i=m,s$, and the following lemma~\cite{Hua2010}:
\begin{enumerate}[label=L.\arabic*]
\item \label{item:L1}~\cite{Hua2010}
For any positive definite matrix $\bm{\Upsilon}$ and vectors $\bm{a}(t)$ and $\bm{b}(\xi)$ with appropriate dimensions, the following inequality holds:
\begin{align*}
\pm 2\tbm{a}(t)\int^{t}_{t-d(t)}\fbm{b}(\xi)d\xi-\int^{t}_{t-d(t)}\tbm{b}(\xi)\bm{\Upsilon}\fbm{b}(\xi)d\xi\\
\leq \overline{d}\cdot\tbm{a}(t)\bm{\Upsilon}^{-1}\fbm{a}(t)\textrm{,}\quad \forall \fbm{a}(t), \fbm{b}(\xi), 0\leq d(t)\leq \overline{d}\textrm{.}
\end{align*}
\end{enumerate}

Using property~\ref{item:P2}, the time derivative of $V_{1}$ along the transformed master and slave dynamics~\eqref{equ8} and their auxiliary dynamics~\eqref{equ31} is:
\begin{equation}\label{equ36}
\begin{aligned}
\dot{V}_{1}=&\sum_{i=m,s}\bigg[\sigma\cdot\tbm{s}_{i}\fbm{\Delta}_{i}-\tbm{s}_{i}\fbm{K}_{i}(\dotbm{q}_{i})\cdot\fbm{s}_{i}-\tbm{s}_{i}\fbm{P}_{i}\cdot(\fbm{q}_{i}-\hbm{q}_{i})\\
&\quad \quad \quad -\tbm{s}_{i}\fbm{D}_{i}\dotbm{q}_{i}-\thbm{s}_{i}\hbm{K}_{i}\hbm{s}_{i}-\thbm{s}_{i}\hbm{e}_{i}-\thbm{s}_{i}\hbm{D}_{i}\dothatbm{q}_{i}\\
&\quad \quad \quad +\hat{\sigma}\cdot\thbm{s}_{i}\fbm{P}_{i}\cdot\left(\dothatbm{q}_{i}-\dotbm{q}_{i}\right)\bigg]+\tbm{s}_{m}\bm{\tau}_{h}+\tbm{s}_{s}\bm{\tau}_{e}\\
&-\thbm{s}_{m}\left(\hat{\sigma}\cdot\hbm{K}_{m}+\fbm{I}\right)\cdot\hbm{P}\cdot(\hbm{q}_{s}-\hbm{q}_{sd})\\
&-\thbm{s}_{s}\left(\hat{\sigma}\cdot\hbm{K}_{s}+\fbm{I}\right)\cdot\hbm{P}\cdot(\hbm{q}_{m}-\hbm{q}_{md})\\
&+\hat{\sigma}\cdot\thbm{s}_{m}\hbm{P}\cdot\left(\dothatbm{q}_{m}-\dothatbm{q}_{s}\right)+\hat{\sigma}\cdot\thbm{s}_{s}\hbm{P}\cdot\left(\dothatbm{q}_{s}-\dothatbm{q}_{m}\right)\\
&+\sum_{i=m,s}\left[\dottbm{q}_{i}\fbm{P}_{i}\cdot(\fbm{q}_{i}-\hbm{q}_{i})+\dothattbm{q}_{i}\fbm{P}_{i}\cdot(\hbm{q}_{i}-\fbm{q}_{i})\right]\\
&+\dothattbm{q}_{m}\hbm{P}\cdot(\hbm{q}_{m}-\hbm{q}_{s})+\dothattbm{q}_{s}\hbm{P}\cdot(\hbm{q}_{s}-\hbm{q}_{m})\textrm{.}
\end{aligned}
\end{equation}

Further, the definition of mismatches $\fbm{\Delta}_{i}$ in~\eqref{equ33} leads to: 
\begin{equation}\label{equ37}
\begin{aligned}
\tbm{s}_{i}\fbm{\Delta}_{i}\leq &\left|\tbm{s}_{i}\fbm{M}_{i}(\fbm{q}_{i})\cdot\left(\dotbm{q}_{i}-\dothatbm{q}_{i}\right)\right|\\
&+\left|\tbm{s}_{i}\fbm{C}_{i}(\fbm{q}_{i},\dotbm{q}_{i})\cdot\left(\fbm{q}_{i}-\hbm{q}_{i}\right)\right|\\
\leq &\lambda_{i2}\cdot\left(\tbm{s}_{i}\fbm{s}_{i}+\frac{1}{2}\dottbm{q}_{i}\dotbm{q}_{i}+\frac{1}{2}\dothattbm{q}_{i}\dothatbm{q}_{i}\right)\\
&+c_{i}\cdot\left[\|\dotbm{q}_{i}\|^{2}\cdot\tbm{s}_{i}\fbm{s}_{i}+\frac{1}{4}(\fbm{q}_{i}-\hbm{q}_{i})^\mathsf{T}(\fbm{q}_{i}-\hbm{q}_{i})\right]\textrm{.}
\end{aligned}
\end{equation}
The sliding surfaces $\fbm{s}_{i}$~\eqref{equ32} imply that:
\begin{equation}\label{equ38}
\begin{aligned}
&\sum_{i=m,s}\left[\dottbm{q}_{i}\fbm{P}_{i}\cdot(\fbm{q}_{i}-\hbm{q}_{i})-\tbm{s}_{i}\fbm{P}_{i}\cdot(\fbm{q}_{i}-\hbm{q}_{i})-\tbm{s}_{i}\fbm{D}_{i}\dotbm{q}_{i}\right]\\
\leq &-\sum_{i=m,s}\bigg[\sigma\cdot(\fbm{q}_{i}-\hbm{q}_{i})^\mathsf{T}\left(\fbm{P}_{i}-\frac{\mu_{i}}{4}\cdot\fbm{D}_{i}\right)\cdot(\fbm{q}_{i}-\hbm{q}_{i})\\
&\quad \quad \quad \quad +\left(1-\frac{\sigma}{\mu_{i}}\right)\cdot\dottbm{q}_{i}\fbm{D}_{i}\dotbm{q}_{i}\bigg]
\end{aligned}
\end{equation}
with $\mu_{i}>0$. Similarly, the sliding surfaces $\hbm{s}_{i}$~\eqref{equ29} imply that:
\begin{equation}\label{equ39}
\begin{aligned}
&\sum_{i=m,s}\bigg[\dothattbm{q}_{i}\fbm{P}_{i}\cdot(\hbm{q}_{i}-\fbm{q}_{i})-\thbm{s}_{i}\hbm{e}_{i}-\thbm{s}_{i}\hbm{D}_{i}\dothatbm{q}_{i}\bigg]\\
&+\dothattbm{q}_{m}\hbm{P}\cdot(\hbm{q}_{m}-\hbm{q}_{s})+\dothattbm{q}_{s}\hbm{P}\cdot(\hbm{q}_{s}-\hbm{q}_{m})\\
\leq &-\sum_{i=m,s}\left[\left(1-\frac{\hat{\sigma}}{\nu_{i}}\right)\cdot\dothattbm{q}_{i}\hbm{D}_{i}\dothatbm{q}_{i}+\hat{\sigma}\cdot\thbm{e}_{i}\left(\fbm{I}-\frac{\nu_{i}}{4}\cdot\hbm{D}_{i}\right)\cdot\hbm{e}_{i}\right]
\textrm{,}
\end{aligned}
\end{equation} 
with $\nu_{i}$ a positive constant, and that:
\begin{equation}\label{equ40}
\begin{aligned}
&\big(\hbm{s}_{m}-\hbm{s}_{s}\big)^\mathsf{T}\hbm{P}\cdot\left(\dothatbm{q}_{m}-\dothatbm{q}_{s}\right)+\sum_{i=m,s}\thbm{s}_{i}\fbm{P}_{i}\cdot\left(\dothatbm{q}_{i}-\dotbm{q}_{i}\right)\\
\leq &\sum_{i=m,s}\zeta_{i}\cdot\thbm{s}_{i}\left(\hbm{P}\hbm{P}+\fbm{P}_{i}\fbm{P}_{i}\right)\cdot\hbm{s}_{i}+\frac{1}{2\zeta_{i}}\cdot\left(3\dothattbm{q}_{i}\dothatbm{q}_{i}+\dottbm{q}_{i}\dotbm{q}_{i}\right)\textrm{,}
\end{aligned}
\end{equation}
with $\zeta_{i}$ another positive constant. 

The derivative of $V_{2}$ is bounded by:
\begin{equation}\label{equ41}
\begin{aligned}
\dot{V}_{2}=&-\gamma\cdot V_{2}+\sum_{i=m,s}\overline{d}_{i}\cdot\dothattbm{q}_{i}\fbm{Q}_{i}\dothatbm{q}_{i}\\
&-\sum_{i=m,s}\int^{t}_{t-\overline{d}_{i}}e^{-\gamma(t-\xi)}\cdot\dothattbm{q}_{i}(\xi)\fbm{Q}_{i}\dothatbm{q}_{i}(\xi)d\xi\\
\leq &-\gamma\cdot V_{2}+\sum_{i=m,s}\overline{d}_{i}\cdot\dothattbm{q}_{i}\fbm{Q}_{i}\dothatbm{q}_{i}\\
&-\sum_{i=m,s}e^{-\gamma\overline{d}_{i}}\cdot\int^{t}_{t-d_{i}}\dothattbm{q}_{i}(\xi)\fbm{Q}_{i}\dothatbm{q}_{i}(\xi)d\xi\textrm{.}
\end{aligned}
\end{equation}
Then, Lemma~\ref{item:L1}~\cite{Hua2010} yields:
\begin{equation}\label{equ42}
\begin{aligned}
&-\thbm{s}_{i}\left(\hat{\sigma}\cdot\hbm{K}_{i}+\fbm{I}\right)\cdot\hbm{P}\cdot\left(\hbm{q}_{j}-\hbm{q}_{jd}\right)\\
&-e^{-\gamma\overline{d}_{j}}\cdot\int^{t}_{t-d_{j}}\dothattbm{q}_{j}(\xi)\fbm{Q}_{j}\dothatbm{q}_{j}(\xi)d\xi\\
\leq &\frac{\overline{d}_{j}}{4\cdot e^{\gamma\overline{d}_{j}}}\cdot\thbm{s}_{i}\left(\hat{\sigma}\cdot\hbm{K}_{i}+\fbm{I}\right)\cdot\hbm{P}\ibm{Q}_{j}\hbm{P}\left(\hat{\sigma}\cdot\hbm{K}_{i}+\fbm{I}\right)\cdot\hbm{s}_{i}
\textrm{,}
\end{aligned}
\end{equation}
where $i,j={m,s}$ and $i\neq j$.

After substitution from~\eqsref{equ37}{equ40}, \eqref{equ42} and~\eqref{equ16}, the sum of~\eqref{equ36} and~\eqref{equ41} leads to:
\begin{equation}\label{equ43}
\begin{aligned}
\dot{V}\leq &-\sum_{i=m,s}\bigg[\tbm{s}_{i}\overline{\fbm{K}}_{i}(\dotbm{q}_{i})\cdot\fbm{s}_{i}+(\fbm{q}_{i}-\hbm{q}_{i})^\mathsf{T}\overline{\fbm{P}}_{i}(\fbm{q}_{i}-\hbm{q}_{i})\\
&\quad \quad \quad \quad +\thbm{s}_{i}\tilde{\fbm{K}}_{i}\hbm{s}_{i}+\hat{\sigma}\cdot\thbm{e}_{i}\cdot\left(\fbm{I}-\frac{\nu_{i}}{4}\cdot\hbm{D}_{i}\right)\cdot\hbm{e}_{i}\\
&\quad \quad \quad \quad +\dottbm{q}_{i}\overline{\fbm{D}}_{i}\dotbm{q}_{i}+\dothattbm{q}_{i}\tilde{\fbm{D}}_{i}\dothatbm{q}_{i}\bigg]\\
&-\gamma\cdot V_{2}+\frac{1}{4}\cdot\left(\frac{\|\bm{\tau}_{h}\|^{2}}{\omega_{m}}+\frac{\|\bm{\tau}_{e}\|^{2}}{\omega_{s}}\right)\textrm{,}
\end{aligned}
\end{equation}
where:
\begin{equation}\label{equ44}
\begin{aligned}
\overline{\fbm{K}}_{i}(\dotbm{q}_{i})=&\fbm{K}_{i}(\dotbm{q}_{i})-\left(\sigma\cdot\lambda_{i2}+\sigma\cdot c_{i}\cdot\|\dotbm{q}_{i}\|^{2}+\omega_{i}\right)\cdot\fbm{I}\textrm{,}\\
\overline{\fbm{P}}_{i}=&\sigma\cdot\fbm{P}_{i}-\frac{\sigma}{4}\cdot\left(c_{i}\cdot\fbm{I}+\mu_{i}\cdot\fbm{D}_{i}\right)\textrm{,}\\
\overline{\fbm{D}}_{i}=&\left(1-\frac{\sigma}{\mu_{i}}\right)\cdot\fbm{D}_{i}-\frac{1}{2}\cdot\left(\sigma\cdot\lambda_{i2}+\frac{\hat{\sigma}}{\zeta_{i}}\right)\cdot\fbm{I}\textrm{,}\\
\tilbm{K}_{i}=&\hbm{K}_{i}-\hat{\sigma}\cdot\zeta_{i}\cdot\left(\hbm{P}\hbm{P}+\fbm{P}_{i}\fbm{P}_{i}\right)\\
&-\frac{\overline{d}_{j}}{4\cdot e^{\gamma\overline{d}_{j}}}\cdot\left(\hat{\sigma}\cdot\hbm{K}_{i}+\fbm{I}\right)\cdot\hbm{P}\ibm{Q}_{j}\hbm{P}\left(\hat{\sigma}\cdot\hbm{K}_{i}+\fbm{I}\right)\textrm{,}\\
\tilbm{D}_{i}=&\left(1-\frac{\hat{\sigma}}{\nu_{i}}\right)\cdot\hbm{D}_{i}-\overline{d}_{i}\cdot\fbm{Q}_{i}-\left(\frac{\sigma\cdot\lambda_{i2}}{2}+\frac{3\hat{\sigma}}{2\zeta_{i}}\right)\cdot\fbm{I}\textrm{.}
\end{aligned}
\end{equation}
Letting $\tilde{\fbm{q}}=\begin{bmatrix}(\fbm{q}_{m}-\hbm{q}_{m})^\mathsf{T}, (\fbm{q}_{s}-\hbm{q}_{s})^\mathsf{T}, (\hbm{q}_{m}-\hbm{q}_{s})^\mathsf{T}\end{bmatrix}^\mathsf{T}$ and using the definition of $\hbm{e}_{i}$ in~\eqref{equ30} lead to:
\begin{equation}\label{equ45}
\begin{aligned}
&\sum_{i=m,s}\Big[(\fbm{q}_{i}-\hbm{q}_{i})^\mathsf{T}\overline{\fbm{P}}_{i}(\fbm{q}_{i}-\hbm{q}_{i})\\
&\quad \quad \quad +\hat{\sigma}\cdot\thbm{e}_{i}\cdot\left(\fbm{I}-\frac{\nu_{i}}{4}\cdot\hbm{D}_{i}\right)\cdot\hbm{e}_{i}\Big]=\ttilbm{q}\tilbm{P}\tilbm{q}\textrm{,}
\end{aligned}
\end{equation}
where $\tilbm{P}=[\fbm{B}_{rc}]$ with $\fbm{B}_{12}=\fbm{B}_{21}=\fbm{0}$ and:
\begin{equation}\label{equ46}
\begin{aligned}
\fbm{B}_{11}=&\fbm{P}_{m}\left(\fbm{I}-\frac{\nu_{m}}{4}\cdot\hbm{D}_{m}\right)\cdot\fbm{P}_{m}+\overline{\fbm{P}}_{m}\textrm{,}\\
\fbm{B}_{13}=&\tbm{B}_{31}=\fbm{P}_{m}\left(\fbm{I}-\frac{\nu_{m}}{4}\cdot\hbm{D}_{m}\right)\cdot\hbm{P}\textrm{,}\\
\fbm{B}_{22}=&\fbm{P}_{s}\left(\fbm{I}-\frac{\nu_{s}}{4}\cdot\hbm{D}_{s}\right)\cdot\fbm{P}_{s}+\overline{\fbm{P}}_{s}\textrm{,}\\
\fbm{B}_{23}=&\tbm{B}_{32}=\fbm{P}_{s}\left(\fbm{I}-\frac{\nu_{s}}{4}\cdot\hbm{D}_{s}\right)\cdot\hbm{P}\textrm{,}\\
\fbm{B}_{33}=&\hbm{P}\left(2\fbm{I}-\frac{\nu_{m}}{4}\cdot\hbm{D}_{m}-\frac{\nu_{s}}{4}\cdot\hbm{D}_{s}\right)\hbm{P}\textrm{.}
\end{aligned}
\end{equation}

\begin{proposition}
Let:
\begin{equation}\label{equ47}
\fbm{P}_{i}=\fbm{P}\textrm{,}\quad \hbm{D}_{i}=\hbm{D}\textrm{,}\quad \text{and}\quad \nu_{i}=\nu\textrm{,}
\end{equation}
and:
\begin{equation}\label{equ48}
\overline{\fbm{P}}_{i}\succ\fbm{0}\textrm{,}\quad \fbm{I}-\frac{\nu}{4}\cdot\hbm{D}\succ\fbm{0}\textrm{,}
\end{equation}
for $i=m,s$. Then there exists $\delta>0$ such that:
\begin{equation}\label{equ49}
\tilbm{P}\succeq \frac{\delta}{2}\cdot\max\left(\fbm{P}, \hbm{P}\right)\textrm{.}
\end{equation}
\end{proposition}
\begin{proof}
It suffices to show that $\tilbm{P}$ is positive definite. By the Schur complement decomposition, $\tilbm{P}\succ\fbm{0}$ if and only if:
\begin{align*}
\fbm{B}_{33}\succ\fbm{0} \quad\textrm{and}\quad
\begin{aligned}
\begin{bmatrix}\fbm{B}_{11} &\fbm{0}\\ \fbm{0} &\fbm{B}_{22}\end{bmatrix}-\begin{bmatrix}\fbm{B}_{13}\\ \fbm{B}_{23}\end{bmatrix}\ibm{B}_{33}\begin{bmatrix}\tbm{B}_{13}\ \tbm{B}_{23}\end{bmatrix}\succ\fbm{0}\textrm{.}
\end{aligned}
\end{align*}
From~\eqref{equ46}-\eqref{equ48}, it follows that $\fbm{B}_{33}\succ\fbm{0}$, that:
\begin{align*}
\begin{aligned}
\begin{bmatrix}\fbm{B}_{13}\\ \fbm{B}_{23}\end{bmatrix}\ibm{B}_{33}\begin{bmatrix}\tbm{B}_{13}\ \tbm{B}_{23}\end{bmatrix}=\frac{1}{2}\cdot\begin{bmatrix}1 & 1\\ 1 & 1\end{bmatrix}\otimes\left[\left(\fbm{I}-\frac{\nu}{4}\cdot\hbm{D}\right)\cdot\fbm{P}^{2}\right]
\textrm{.}
\end{aligned}
\end{align*}
and, further, that:
\begin{align*}
\begin{aligned}
&\begin{bmatrix}\fbm{B}_{11} &\fbm{0}\\ \fbm{0} &\fbm{B}_{22}\end{bmatrix}-\begin{bmatrix}\fbm{B}_{13}\\ \fbm{B}_{23}\end{bmatrix}\ibm{B}_{33}\begin{bmatrix}\tbm{B}_{13}\ \tbm{B}_{23}\end{bmatrix}\\
=&\begin{bmatrix}\overline{\fbm{P}}_{m} & \fbm{0}\\ \fbm{0} & \overline{\fbm{P}}_{s}\end{bmatrix}+\frac{1}{2}\cdot\begin{bmatrix}1 & -1\\ -1 & 1\end{bmatrix}\otimes\left[\left(\fbm{I}-\frac{\nu}{4}\cdot\hbm{D}\right)\cdot\fbm{P}^{2}\right]\succ\fbm{0}\textrm{.}
\end{aligned}
\end{align*}
\end{proof}

\begin{theorem}\label{theorem2}
The teleoperator~\eqref{equ1} in closed-loop with~\eqref{equ28} and~\eqref{equ34} is ISS with input $\fbm{u}=[\tfbm{\tau}_{h}\ \tfbm{\tau}_{e}]^\mathsf{T}$ and state $\fbm{x}=[\dottbm{q}_{m}\ \dottbm{q}_{s}\ \dothattbm{q}_{m}\ \dothattbm{q}_{s}\ (\fbm{q}_{m}-\hbm{q}_{m})^\mathsf{T}\ (\fbm{q}_{s}-\hbm{q}_{s})^\mathsf{T}\ (\hbm{q}_{m}-\hbm{q}_{s})^\mathsf{T}]^\mathsf{T}$ if the parameters and control gains satisfy conditions~\eqref{equ47}-\eqref{equ48} and:
\begin{equation}\label{equ50}
\overline{\fbm{K}}_{i}(\dotbm{q}_{i})\succeq \frac{\psi}{2}\cdot\lambda_{i2}\cdot\fbm{I}\textrm{,}\quad \tilbm{K}_{i}\succeq \frac{\psi}{2}\cdot\fbm{I}\textrm{,}\quad \overline{\fbm{D}}_{i}\succeq\fbm{0}\textrm{,}\quad \tilbm{D}_{i}\succeq\fbm{0}\textrm{,}
\end{equation}
where $i=m,s$ and $\psi>0$.
\end{theorem}

\begin{proof}
The definitions of $\hbm{s}_{i}$ and $\fbm{s}_{i}$, $i=m,s$, imply that:
\begin{equation}\label{equ51}
\begin{aligned}
\|\dothatbm{q}_{i}\|^{2}\leq &2\cdot\|\hbm{s}_{i}\|^{2}+2\hat{\sigma}^{2}\cdot\|\hbm{e}_{i}\|^{2}\textrm{,}\\
\|\hbm{s}_{i}\|^{2}\leq &2\cdot\|\dothatbm{q}_{i}\|^{2}+2\hat{\sigma}^{2}\cdot\|\hbm{e}_{i}\|^{2}\textrm{,}
\end{aligned}
\end{equation}
and that:
\begin{equation}\label{equ52}
\begin{aligned}
\|\dotbm{q}_{i}\|^{2}\leq &2\cdot\|\fbm{s}_{i}\|^{2}+2\sigma^{2}\cdot\|\fbm{q}_{i}-\hbm{q}_{i}\|^{2}\textrm{,}\\
\|\fbm{s}_{i}\|^{2}\leq &2\cdot\|\dotbm{q}_{i}\|^{2}+2\sigma^{2}\cdot\|\fbm{q}_{i}-\hbm{q}_{i}\|^{2}\textrm{.}
\end{aligned}
\end{equation}
Because $V_{2}\geq 0$ and
\begin{equation}\label{equ53}
\|\hbm{e}_{i}\|^{2}\leq 2P^{2}\cdot\|\fbm{q}_{i}-\hbm{q}_{i}\|^{2}+2\hat{P}^{2}\cdot\|\hbm{q}_{m}-\hbm{q}_{s}\|^{2}\textrm{,}
\end{equation}
where $P$ and $\hat{P}$ are the maximum eigenvalue of $\fbm{P}$ and $\hbm{P}$, the Lyapunov candidate $V$ can be lower-bounded by:
\begin{equation}\label{equ54}
\begin{aligned}
V\geq &\frac{1}{2}\sum_{i=m,s}\left(\lambda_{i1}\cdot\|\fbm{s}_{i}\|^{2}+\hat{\lambda}_{i1}\cdot\|\hbm{s}_{i}\|^{2}+p\cdot\|\fbm{q}_{i}-\hbm{q}_{i}\|^{2}\right)\\
&+\frac{\hat{p}}{2}\cdot\|\hbm{q}_{m}-\hbm{q}_{s}\|^{2}\geq\hat{\alpha}_{1}(\|\fbm{x}\|)\textrm{,}
\end{aligned}
\end{equation}
where $p$ and $\hat{p}$ are the minimum eigenvalue of $\fbm{P}$ and $\hbm{P}$, and $\hat{\alpha}_{1}(\|\fbm{x}\|)=a_{1}\cdot\|\fbm{x}\|^{2}$ is a function of class $\mathcal{K}_{\infty}$ with:
\begin{align*}
\frac{1}{a_{1}}=\sum_{i=m,s}\left(\frac{4}{\lambda_{i1}}+\frac{4}{\hat{\lambda}_{i1}}\right)+\frac{4\sigma^{2}+8\hat{\sigma}^{2}P^{2}}{p}+\frac{16\hat{\sigma}^{2}\hat{P}^{2}}{\hat{p}}\textrm{.}
\end{align*}
Further, the definition of $V_{2}$ indicates that:
\begin{equation}\label{equ55}
\begin{aligned}
V_{2}\leq &\sum_{i=m,s}\overline{Q}_{i}\cdot|\dothatbm{q}_{i}|^{2}_{r}\int^{0}_{-\overline{d}_{i}}\int^{t}_{t+\theta}e^{-\gamma(t-\xi)}d\xi d\theta\\
\leq &\sum_{i=m,s}\frac{1}{2}\cdot\overline{d}^{2}_{i}\overline{Q}_{i}\cdot|\dothatbm{q}_{i}|^{2}_{r}\textrm{,}
\end{aligned}
\end{equation}
where $\overline{Q}_{i}$ is the maximum eigenvalue of $\fbm{Q}_{i}$. Then $V$ can be upper-bounded by:
\begin{equation}\label{equ56}
\begin{aligned}
V\leq &\frac{1}{2}\sum_{i=m,s}\left(\lambda_{i2}\cdot\|\fbm{s}_{i}\|^{2}+\hat{\lambda}_{i2}\cdot\|\hbm{s}_{i}\|^{2}+P\cdot\|\fbm{q}_{i}-\hbm{q}_{i}\|^{2}\right)\\
&+\frac{\hat{P}}{2}\cdot\|\hbm{q}_{m}-\hbm{q}_{s}\|^{2}+V_{2}\leq \hat{\alpha}_{2}(|\fbm{x}|_{r})\textrm{,}
\end{aligned}
\end{equation}
where $\hat{\alpha}_{2}(|\fbm{x}|_{r})=a_{2}\cdot|\fbm{x}|^{2}_{r}$ with
\begin{align*}
a_{2}=\max\limits_{i=m,s}\Bigg[\max\Bigg(&\lambda_{i2},\  \hat{\lambda}_{i2}+\frac{\overline{d}^{2}_{i}\overline{Q}_{i}}{2},\ 2\hat{\lambda}_{i2}\hat{\sigma}^{2}\hat{P}^{2}+\frac{\hat{P}}{2},\\
&\lambda_{i2}\sigma^{2}+2\hat{\lambda}_{i2}\hat{\sigma}^{2}P^{2}+\frac{P}{2}\Bigg)\Bigg]\textrm{.}
\end{align*}
Let $\gamma_{a}=\sqrt{a_{1}}$ and $\overline{\gamma}_{a}=\sqrt{a_{2}}$, define $|\fbm{x}_{t}|_{a}=\sqrt{V(\fbm{x}_{t})}$ to satisfy~\eqref{equ2}, and select functions $\alpha_{1}(\|\fbm{x}\|)=\hat{\alpha}_{1}(\|\fbm{x}\|)$ and $\alpha_{2}(|\fbm{x}_{t}|_{a})=|\fbm{x}_{t}|^{2}_{a}$ of class $\mathcal{K}_{\infty}$ to trivially guarantee condition a) of Theorem~\ref{item:T2}.

After substitution from~\eqref{equ45} in~\eqref{equ43}, using condition~\eqref{equ50} and setting $\kappa=\min\left(\psi, \delta, \gamma\right)$, $\dot{V}$ can be upper-bounded by:
\begin{equation}\label{equ57}
\begin{aligned}
\dot{V}\leq -\kappa\cdot V+\frac{1}{4\omega}\cdot\|\fbm{u}\|^{2}\textrm{.}
\end{aligned}
\end{equation}
Then the functions $\alpha_{3}$ and $\rho$ of class $\mathcal{K} $ and defined by:
\begin{equation}\label{equ58}
\alpha_{3}(|\fbm{x}_{t}|_{a})=\frac{\kappa}{2}\cdot|\fbm{x}_{t}|^{2}_{a}\textrm{,}\ \text{and}\ \rho(\|\fbm{u}\|)=\sqrt{\frac{1}{2\kappa\omega}}\cdot\|\fbm{u}\|
\end{equation}
ensure condition b) of Theorem~\ref{item:T2}~\cite{Jiang2006}.

Thus, the Lyapunov candidate $V$ obeys both conditions of Theorem~\ref{item:T2}~\cite{Jiang2006} and the teleoperator is ISS.
\end{proof}

\begin{corollary}\label{corollary2}
ISS teleoperation renders invariant the set:
\begin{equation}\label{equ59}
\mathcal{S}_{I}=\left\{\fbm{q}_{m}-\fbm{q}_{s}:\ \|\fbm{q}_{m}-\fbm{q}_{s}\|^{2}\leq \frac{4}{p'}\cdot\left(V_{0}+\frac{\overline{\tau}^{2}}{4\kappa\omega}\right)\right\}\textrm{,}
\end{equation}
and globally attractive the set:
\begin{equation}\label{equ60}
\mathcal{S}_{A}=\left\{\fbm{q}_{m}-\fbm{q}_{s}:\ \|\fbm{q}_{m}-\fbm{q}_{s}\|^{2}\leq \frac{\overline{\tau}^{2}}{p'\kappa\omega}\right\}\textrm{,}
\end{equation}
where $p'=\min(p_{m},p_{s},\hat{p})$.
\end{corollary}

\begin{proof}
Time integration of~\eqref{equ57} leads to:
\begin{equation}\label{equ61}
V(t)\leq e^{-\kappa t}\cdot V(0)+\frac{\overline{\tau}^{2}}{4\kappa\omega}\textrm{,}
\end{equation}
while the definition of $V$ implies that
\begin{equation}\label{equ62}
\begin{aligned}
V\geq &\sum_{i=m,s}\frac{p_{i}}{2}\cdot\|\fbm{q}_{i}-\hbm{q}_{i}\|^{2}+\frac{\hat{p}}{2}\cdot\|\hbm{q}_{m}-\hbm{q}_{s}\|^{2}\\
\geq &\frac{p'}{4}\cdot\|\fbm{q}_{m}-\fbm{q}_{s}\|^{2}\textrm{,}
\end{aligned}
\end{equation}
and completes the proof.
\end{proof}

\subsection{Discussion}

Closed-loop bilateral teleoperation includes uncertain and practically uncontrollable user and environment dynamics. To robustly stabilize closed-loop teleoperation, this paper regards the user and environment forces as time-varying and unpredictable teleoperator inputs, and designs controllers to render the teleoperator input-to-state stable. ISS teleoperation guarantees robust position synchronization of the master and slave under user and environment perturbations~\cite{Jiang2008}, and thus, reflects slave-environment interactions to the operator~\cite{Heck2018}. According to the definitions in~\sect{sec: iss}, ISS teleoperation implies: (i) $\{\dotbm{q}_{m},\dotbm{q}_{s},\fbm{q}_{m}-\fbm{q}_{s}\}\in\mathcal{L}_{\infty}$; and (ii) $\{\dotbm{q}_{m},\dotbm{q}_{s},\fbm{q}_{m}-\fbm{q}_{s}\}\to\fbm{0}$ if $\bm{\tau}_{h}=\bm{\tau}_{e}=\fbm{0}$. Thus, ISS teleoperation matches the stability definition in~\cite{Lee2010} for bilateral teleoperation.

According to~\coro{corollary1} and~\coro{corollary2}, an invariant set~$\mathcal{S}_{I}$ and a globally attractive set~$\mathcal{S}_{A}$ characterize the master-slave position error $\fbm{q}_{m}(t)-\fbm{q}_{s}(t)$ of ISS teleoperation: the error stays in~$\mathcal{S}_{I}$ and exponentially approaches~$\mathcal{S}_{A}$ for $t\geq 0$. The Lebesgue measure of $\mathcal{S}_{I}$ and $\mathcal{S}_{A}$ can be reduced by increasing $p_{i}$, $\hat{p}$, $\omega$ and $\kappa$. The speed of~(exponential) convergence to~$\mathcal{S}_{A}$ is determined by $\kappa$. When the user and environment forces disppear, the steady-state position error becomes zero by the definition of input-to-state stability.

The controllers~\eqref{equ10} and~\eqref{equ34} assume gravity-compensated master and slave. Inaccurate gravity compensation can introduce bounded disturbances $\bm{\delta}_{i}$, $i=m,s$, in practical teleoperation systems. By including $\bm{\delta}_{i}$ in the teleoperator input $\fbm{u}^{\star}=[(\bm{\tau}_{h}+\bm{\delta}_{m})^\mathsf{T}\ (\bm{\tau}_{e}+\bm{\tau}_{s})^\mathsf{T}]^\mathsf{T}$, the proposed dynamic strategy can render the teleoperator ISS with the augmented input $\fbm{u}^{\star}$ and the same state $\fbm{x}$ as defined in Theorem 1 and Theorem 2. Alternatively, minor modifications of adaptive control techniques~[26] can be employed in the design to provide parameter estimates and to facilitate position synchronization.

\section{Conclusions}

This paper has presented a constructive dynamic interconnection and damping injection strategy for robust stabilization of bilateral teleoperation without, and with, time-varying delays. Lyapunov stability analysis has proven that the proposed strategy renders bilateral teleoperation exponentially input-to-state stable, even in the presence of time-varying delays. It has also shown that an invariant set and a globally attractive set characterize the master-slave position errors during teleoperation under the control of the proposed strategy. Suitable selection and updating of the control gains can decrease the master-slave position error to any prescribed level with a certain rate of convergence and, thus, can improve robust position tracking performance. 


Practical teleoperation systems suffer from inaccurate gravity compensation and unreliable velocity measurements. Disturbances caused by inaccurate gravity compensation can increase position tracking errors. Unreliable velocity measurements impede damping injection and the modulation of control gains and thus, can destabilize the teleoperation. Future work will focus on input-to-state stability of bilateral teleoperators without gravity compensation and velocity measurements.

\bibliography{IEEEabrv,Cooperation}
\end{document}